\documentclass[11pt]{amsart}
\usepackage[utf8x]{inputenc}

\usepackage{dsfont}
\usepackage{amsmath}
\usepackage{amssymb}
\usepackage{graphicx}
\usepackage{amsfonts}
\usepackage{soul}
\usepackage{mathpazo}
\usepackage{color}
\usepackage{yfonts}
\usepackage{paralist}
\usepackage{stmaryrd}
\usepackage{amsxtra}
\usepackage{latexsym,mathrsfs}
\usepackage{mathabx}
\usepackage[a4paper, margin=1in]{geometry}

\newtheorem{theorem}{Theorem}
\newtheorem*{lemma*}{Claim}
\newtheorem{lemma}[theorem]{Lemma}
\newtheorem{coro}[theorem]{Corollary}
\newtheorem{definition}[theorem]{Definition}

\newtheorem{prop}[theorem]{Proposition}
\newtheorem{Remark}[theorem]{Remark}

\numberwithin{theorem}{section}
\numberwithin{equation}{section}

\title{A Note on Weak Saturation Number of Trees}

\author{Wenchong Chen}
\address[Chen]{Nankai University, Weijin Road 94, Nankai District, Tianjin, 300192, PR China}
\email{2212161@mail.nankai.edu.cn}

\author{Xiao-Chuan Liu}
\address[Liu]{Departamento de Matemática,
 Universidade Federal de Pernambuco,
	Avenida Jornalista Aníbal Fernandes, Cidade Universitária, Recife, Brazil}
\email{xiaochuan.liu@ufpe.br}

\author{Xu Yang}
\address[Yang]{Instituto de Computação, Universidade Federal de Alagoas,
	Av. Lourival Melo Mota, S/N, Maceió, Brazil} 
\email{yang@ic.ufal.br}

\begin{document}

\maketitle

\begin{abstract}In this paper, we estimate the weak saturation numbers of trees. 
As a case study, we examine  caterpillars and obtain several tight estimates. In particular, this implies 
that for any $\alpha\in [1,2]$, there exist caterpillars with $k$ vertices whose weak saturation numbers are of order $k^\alpha$.
We call a tree good if its weak saturation number is exactly its edge number minus one. 
We provide a sufficient condition for a tree to be a good tree. 
With the additional property that all leaves are at even distances from each other, this condition fully characterizes 
good trees. 
\end{abstract}

\section{Introduction}
\label{sec:introduction}

We follow ~\cite{murty2008graph} for graph theory notation. Throughout, we consider only simple, finite, and undirected graphs. We use 
$V(G)$ and $E(G)$ to denote the vertex set and edge set of a graph $G$, respectively, and $v(G)$ and
$e(G)$ to denote the number of vertices and edges of $G$.
Given a graph $F$, a graph $G$ is said to be weakly 
$F$-saturated if its non-edges can be added one by one in a specific order, such that each newly added edge forms a new copy of $F$.
For any integer  $n$, the weak saturation number of  $F$, denoted  $\text{w-sat}(n, F)$, 
is the minimum size of an $n$-vertex weakly $F$-saturated graph. 
This process produces a sequence of graphs, each with one additional edge, culminating in the complete graph $K_n$. This process is also referred to as an $F$-bootstrap percolation process. Bollobás~\cite{bollobas1968weakly} introduced both notions in 1968 and conjectured that  
$\text{w-sat}(n,K_k)={k\choose 2}-1+(n-k)(k-2)$ for any $n\geq k$, a result later proved by Lov\'asz in~\cite{lovasz1977flats}. 
Lov\'asz's proof involves flats in matroids representable over fields and is somewhat mysterious, given that this is a purely combinatorial problem. 
Consequently, several researchers have studied these problems, offering various proofs, including those by Frankl~\cite{frankl1982extremal}, Kalai~\cite{kalai1984weakly}, Alon~\cite{alon1985extremal}, and Yu~\cite{yu1993extremal}. It remains an open question whether a truly elementary combinatorial proof exists, as none of the above qualify. Determining the exact values of weak saturation numbers is generally very difficult. 
Aside from the complete graphs discussed above, 
a recent result states that $\text{w-sat}(n,K_{t,t})={t\choose 2} + (t-1)(n-t/2+1)$, with $t\geq 2$ and $n\geq 3t-3$, which was proved in~\cite{kronenberg2021weak}
using the elegant linear algebra method. 
We also highlight a recent combinatorial approach by Terekhov and Zhukovskii~\cite{terekhov2023weak}, 
which provides precise weak saturation numbers for certain graphs.

Turning our attention to a more relaxed problem, we can ask 
for the correct order and leading coefficients for the weak saturation numbers of some fixed graph 
$F$ with $k$ vertices and with minimum degree  $\delta$. 
A straightforward upper bound
arises from constructing a copy of $K_k\setminus e$ along with  $n-k$ additional vertices, 
each joining to $\delta-1$ vertices of that copy. 
This immediately implies that for any $n\geq v(F)$, 
$\text{w-sat}(n,F)\leq {k\choose 2}-1+(\delta-1)(n-k)=(\delta-1)n +O(1)$. However, there are no matching lower bounds in general. 
Faudree, Gould and Jacobson originally claimed the estimate $\text{w-sat}(n,F)\geq (\frac {\delta}{2}-\frac{1}{\delta+1} )n +o(n)$ in~\cite{faudree2013weak}, though their proof contained flaws. The result was later corrected by Terekhov and Zhukovskii in~\cite{terekhov2023weak}.
In this paper, we mainly discuss weak saturation number of trees. 
Likely because their weak saturation numbers are  $O(1)$, trees, despite being the simplest and most fundamental types of graphs, have not been fully understood and were somewhat overlooked, with only a few results available, as seen in ~\cite{faudree2011survey, faudree2013weak}. 
First, we note that the number 
$\text{w-sat}(n,T)$ is always non-increasing. The proof will be provided in Section~\ref{sec:General_Estimation}.
\begin{prop}\label{decrease}
For any graph $F$ with minimum degree $1$, $\text{w-sat}(n, F)$ is 
non-increasing with respect to $n$ for $n\geq v(F)$.   
\end{prop}

Therefore, for a tree $T$, $\text{w-sat}(n,T)$
remains constant for all sufficiently large $n$.
We denote by $\text{w-sat}(T)$ the limit number of 
$\text{w-sat}(n,T)$ as $n$ tends to infinity.
In extremal graph theory, 
a famous rational exponent conjecture  due to Erd\H{o}s and Simonovits~\cite{erdHos1981combinatorial} 
states that,
for every rational number  $r \in (1, 2)$, there exists a graph whose Tur{\'a}n exponent equals  $r$. 
However, for trees, the weak saturation number depends only on their size $K$ and not on $n$.
In this paper, we prove 
that all rational exponents between $1$ and $2$ 
can be realized by  weak saturation numbers of trees
in terms of their orders.

\begin{theorem}\label{exponents} 
For any $\alpha \in [1,2]$, 
	there exists a family $\{T_K\}_{K\geq 5}$ of trees such that, each $T_K$ has $K$ vertices,  and 
	$\text{w-sat}(T_K)=\Theta(K^{\alpha})$. 
\end{theorem}

A nondegenerate caterpillar 
$C_{a_1, a_2,\dots, a_\ell}$, where $a_t>0$, $t=1,2,\dots, \ell$, is formed by adding $a_t$ 
pendant edges to each $v_t$ of a path 
$P_\ell=(v_1,v_2,\dots,v_\ell)$, for 
every $t=1,\dots,\ell$. 
Our next result gives a precise characterization of the weak saturation numbers for nondegenerate caterpillars.

\begin{theorem}\label{goodcat}
Consider a nondegenerate 
caterpillar $T=C_{a_1, a_2, \dots, a_\ell}$ on $K$ vertices, 
and write $a=\min \{a_t  \, \big| \ 1\leq t\leq \ell \}$. Then $\text{w-sat}(T)=K-2+{a-1\choose 2}$.
\end{theorem}

A tree $T$ is called a {\it good tree}
if $\text{w-sat}(T)=e(T)-1$. 
It is natural to explore the structure of good trees, and we prove the next theorem for this purpose. 
\begin{theorem}\label{goodtree}
 Suppose a tree $T$ satisfies that the distances between any two leaves are even.
Then $T$ is a good tree 
if and only if there exists a leaf that is adjacent to a vertex of degree 2.
\end{theorem}

The rest of the paper is organized as follows. In Section~\ref{sec:General_Estimation}, we will present some general estimates for trees and more generally graphs with at least one pendant edge. In 
Section~\ref{sec:Caterpillar}, we focus on caterpillars and prove Theorems~\ref{exponents} and \ref{goodcat}.
In Section~\ref{sec:Local_Structures} we prove Theorem~\ref{goodtree}. We will also present a counterexample of Theorem 8 in~\cite{faudree2013weak}. 

\section{General Estimates}
\label{sec:General_Estimation}

In this section, we present several general estimates for weak saturation numbers as well as the 
proof of Proposition~\ref{decrease}.
We will use the following definition introduced in~\cite{faudree2013weak}.

\begin{definition}
Given a graph $F$, we define an induced subgraph $S=K_{1,s}$ to be an end-star of $F$, 
if the center of $S$ has degree $s+1$ in $F$ and $s$ of the other vertices of $S$ have degree 1 in $F$. 
If $F$ contains at least an end-star, 
the minimum end-degree of $F$, denoted by $\delta_e (F)$, 
is the minimum degree of the center of an end-star of $F$.    
\end{definition}

The following simple lemma turns out to be crucial.
\begin{lemma}\label{min_degree_condition}
Let $F$ be a graph on $K$ vertices with minimum degree 1.  
Suppose a graph $G$ with $v(G)\geq K$ 
admits 
an induced complete subgraph $H$ on at least $K-1$ vertices.  
Then $G$ is weakly $F$-saturated. 
\end{lemma}

\begin{proof}
Let $F_1$ be a subgraph of $F$ by deleting a leaf $u$ and let $\{u'\}=N_F(u)$. By hypothesis $G$ contains an induced complete subgraph $H$ with $|V(H)|\geq K-1$, so $H$ contains a copy of $F_1$. Moreover, because $H$ is complete, for every vertex $w\in V(H)$ there is an embedding of $F_1$ into $H$ that send $u'$ to $w$.  Now, it suffices to show we can iteratively add edges to the graph $G$, 
until it becomes the complete graph $K_n$. 

\begin{enumerate}[(a)]
\item For any 
    $w\in V(H)$ and $v \in V(G)\setminus V(H)$, 
    we can add the edge $vw$. A new copy of $F$ is created such that $uu'$ is mapped to $vw$. 
\item Now for any 
    $v \in V(G)\setminus V(H)$, $V(H)\subseteq N(v)$, which implies 
    $G[V(H)\cup \{v\}]$ is an induce complete subgraph.
    Then we add all the remaining edges one at a time, 
each of which will create a copy 
    while mapping $uu'$ to the newly added edge. 
    \end{enumerate}
\end{proof}
Proposition~\ref{decrease} follows directly from  Lemma \ref{min_degree_condition}. 
\begin{proof}[Proof of Proposition~\ref{decrease}]
We prove that $\text{w-sat}(n,F)\leq \text{w-sat}(m,F)$ for any $n > m \geq v(F)$. Let $H$ be a weakly  $F$-saturated graph on $m$ vertices.
Define $G=H+\overline{K}_{n-m}$ so that  
$G$ is also weakly $F$-saturated. 
First, one can add edges to $H$ until it forms $K_m$. Then, by Lemma ~\ref{min_degree_condition}, $G$ is weakly $F$-saturated.
\end{proof}

\begin{lemma}\label{d_endstar} (Theorem 9 in~\cite{faudree2013weak})
Let $F$ be a connected graph with $K$ vertices, containing 
at least one end-star.
For any $n$ sufficiently large, 
\begin{equation}
\text{w-sat}(F)\leq K-2+{{\delta_e(F)}\choose 2}.
\end{equation}
\end{lemma}

We obtain an upper bound estimate that, in certain cases, improves upon the result of the above lemma.

\begin{prop}\label{endd-mind}
Let $F$ be a graph on $K$ vertices containing at least an end-star. Let $u$ and $w$ be adjacent non-pendent vertices in $F$. Let $S$ be the union of the neighbors of $u$ and $w$ that are neither adjacent to any leaf nor equal to $u$ or $w$. Then, we have $\text{w-sat}(F)\leq K-2+|S|\delta_e$.
\end{prop}

\begin{proof}
Let $U=\{u_1, u_2,\dots, u_{\delta_e}\}$ be a set of $\delta_e$ vertices, and let $K_{US}$ be a complete bipartite graph with vertex set $U\cup S$. Let $e$ be any edge in $F$. We construct the graph $G_1$ as the union of $F\setminus e$ and $K_{US}$. Then let $G=G_1+\overline{K}_{n-K-|S|\delta_e}$, where $n\geq 2K+|S|\delta_e$. We will prove that $G$ is $F$-weakly saturated by adding edges to $G$ in the following order:
\begin{enumerate}[(a)]
\item Add $e$.
\item For any neighbor of $u$ or $w$ that is adjacent to a leaf, namely $v$, add $vu_k$ for every $1\leq k\leq \delta_e$ to replace the leaf.
\item Now, for every $1\leq k\leq \delta_e$, $u_k$ is adjacent to all the neighbors of $u$ and $w$ (except $u$ and $w$). We can add $u_iu_j$ to replace $uw$ for each $1\leq i<j\leq \delta_e$.
\item Let $v_1$ be the isomorphic image of the center of the smallest end-star of $F$. We can add $v_1u_k$ for $1\leq k\leq \delta_e$ to replace a pendent edge at $v_1$.
\item Let $v_2$ be the only initially non-leaf neighbor of $v_1$. We can add $v_2u_i$ for any $1\leq i\leq \delta_e$ to replace $v_1$ by $u_i$, and the pendent vertices at $v_1$ by $v_1$ and $u_j$ ($j\neq i$).
\item For any initially isolated vertex $p\in V(\overline{K}_{n-K-|S|\delta_e})$, add $pu_i$, $1\leq i\leq \delta_e$ to replace $v_1$ by $u_i$, and the pendent vertices at $v_1$ by $p$ and $u_j$ ($j\neq i$).
\item For any $p_1, p_2\in V(\overline{K}_{n-K-|S|\delta_e})$, add $p_1p_2$ to replace $v_1$ by $p_1$, and the pendent vertices at $v_1$ by $u_i$ ($1\leq i\leq\delta_e$).
\item Now we have obtained a sufficiently large complete graph. By Lemma~\ref{min_degree_condition}, we can continue adding edges until we obtain $K_n$.
\end{enumerate}
\end{proof}

\begin{Remark}
Steps (d) to (h) clarify the proof of Lemma~\ref{d_endstar}.	
\end{Remark}

\section{A Case Study: Caterpillars}
\label{sec:Caterpillar}
Fix nonnegative integers $a_1, a_2,\ldots,a_\ell$, 
and define a caterpillar $C_{a_1, a_2,\ldots, a_\ell}$ as  
the tree obtained by adding $a_t$ pendant edges to the $t$-th vertex of a path $P_\ell$, where $1\leq t\leq \ell$. 
In this section, we estimate weak saturation numbers for 
caterpillars, providing tight results in many situations. As an application, we also prove Theorem~\ref{exponents}.
We begin with a useful lower bound estimate.

\begin{prop}\label{catwsat}
Consider a caterpillar $T=C_{a_1, a_2,\ldots, a_\ell}$ on $K$ vertices.
Suppose that $a_k+a_{k+1}>0$ for all $1\leq k\leq \ell-1$.
Define $a = \min \{  a_t \big|  1\leq t\leq \ell \text{ and } a_t\neq 0\}$.
Then, 
$$\text{w-sat}(T)\geq K-2+{{a-1}\choose 2}.$$
\end{prop}

\begin{proof} As $\text{w-sat}(T)\geq e(T)-1=K-2$, we may assume that $a\geq 3$.
List all non-leaf vertices of $T$, namely $\{v_1,\ldots,v_\ell\}$.
Define $Q=\{v_t|a_t>0\}$ and let $q=|Q|$.
Consider any weakly $T$-saturated graph $G$ with at least $a+q-1$ vertices. 
Fix an isomorphic copy of $T\setminus e$ in $G$, and 
define $V_1$ as the image of $Q$, while $V_2=V(G)\setminus V_1$. Then $|V_2|$ contains at least $a-1$ vertices. If $V_2$ is a complete graph, the proof is complete. Otherwise, we make the following claim. 

\noindent{\bf Claim}. 
There exists a vertex $u_1\in V_2$ such that $d_{G[V_2]}(u_1)\geq a-2$.
\begin{proof}[Proof of Claim]
Consider the $T$-bootstrap percolation process starting from $G$. 
Let $e'$ be the first edge added with both endpoints in $V_2$, creating a copy of $T$.
Since $a_k+a_{k+1}>0$ holds for all $1\leq k\leq \ell-1$, we deduce that each edge of $T$ has at least one endpoint in $Q$. Therefore, when creating the new copy, at least one element of $Q$, as an endpoint of $e'$, is mapped to $V_2$. 
Further, we assume that exactly $q_1>0$ 
elements of $Q$ are mapped to $V_2$, and that 
these vertices are adjacent to a total of $s$ leaves in the copy.
The remaining $q-q_1$ elements of $Q$
are mapped to $V_1$. Since $|V_1|=q$, at most $q_1$ vertices in the copy can be leaves. Thus, at least $s-q_1$ leaves are in $V_2$, implying that  in $G[V_2]$, there exist $q_1$ vertices collectively adjacent to at least $s-q_1$ edges. Therefore, some vertex $u_1\in V_2$
satisfies:
$$d_{G[V_2]}(u_1) \geq \frac{s-q_1}{q_1} \geq \frac{q_1a-q_1}{q_1}= a-1.$$  
Before adding $e'$, we have $d_{G[V_2]}(u_1)\geq a-2$, proving the claim. 
\end{proof}

Continuing with the proof, we iteratively define vertices  
$u_k$ and $V_{2,k}=V_{2,k-1}-\{u_k\}$ such that $d_{G[V_{2, k-1}]}(u_{k})\geq a-1-k$, 
where $k=1,2,\dots,a-2$ and $V_{2,0}=V_2$. Consequently, the number of edges in $G[V_2]$
is at least 
$e(G[V_2])\geq (a-2)+(a-3)+\dots+1={{a-1}\choose 2}.$

Since $a_k+a_{k+1}>0$ for every $k=1,\dots,\ell-1$, 
every edge in $T$ has at least one endpoint in $Q$. 
It follows that 
$e(G[V_1])+e(G[V_1, V_2])\geq e(T)-1=K-2$, leading to the total edge count: 
$e(G[V_1])+e(G[V_1,V_2])+e(G[V_2])\geq K-2+{{a-1}\choose 2}.$
\end{proof}

\begin{Remark} In ~\cite{faudree2013weak}, Theorem 21,
 items (1) and (4) provide matching upper bounds in the following cases:
\begin{enumerate}
\item $a_j=a$ and both $a_{j-1},a_{j+1}> 0$, or
\item $a_1=a$ and $a_2 > 0$, or $a_{\ell}=a$ and $a_{\ell-1}>0$.
\end{enumerate}
Thus, it follows that 
$\text{w-sat}(C_{a_1,\dots,a_\ell})= K-2 +{a-1\choose 2}$. 
\end{Remark}

Moreover, we are ready to prove Theorem~\ref{goodcat}. 
\begin{proof}[Proof of Theorem~\ref{goodcat}] For (1), if $T$ is a good tree, the result follows from Proposition~\ref{catwsat}. If $a\leq 2$, it follows from Theorem 21 (1) and (4) in \cite{faudree2013weak}. 

For (2), the lower bound is a direct consequence of Proposition~\ref{catwsat}, while the upper bound follows from Theorem~21 (1) and (4) in~ \cite{faudree2013weak}.
\end{proof}

We can now present the proof of Theorem~\ref{exponents}. 
\begin{proof}[Proof of Theorem~\ref{exponents}]
Fix $\alpha\in [1,2]$. 
For sufficiently large $K$, set $a_K=\lfloor \frac{1}{3}K^{\frac{\alpha}{2}} \rfloor$ so that $a_K>2$ and $K-a_K-2>a_K$. Define the tree 
 $T_K=C_{a_K, K-a_K-2}$ on $K$ vertices.
 Then
$\text{w-sat}(T_K)=K-2+{a_K-1\choose 2}=\Theta(K^\alpha)$.   
\end{proof}

\begin{Remark}\label{extend} Recall that an internal tree is the subgraph of a tree obtained by deleting its leaves.
This concept allows us to generalize the above results as follows.

For any tree $T$, let $I(T)$ denote the internal tree of $T$, and let $U$ be the set of vertices in $T$ adjacent to the leaves of $T$. For each $v\in U$, let $d_\ell(v)$ be the number of leaves in its  neighborhood, and define $d=\min \{ d_\ell(v) \, \big| v \in U \}$.
\begin{enumerate}[(a)]
\item If $U$ dominates $I(T)$, then $\text{w-sat}(T)\geq v(T)-2+{{d-1}\choose 2}.$
\item If $V(I(T))=U$, then $\text{w-sat}(T)=v(T)-2+{{d-1}\choose 2}.$
\end{enumerate}
The proof is very similar to that of Theorem~\ref{goodcat} and Proposition~\ref{catwsat}, so we omit it.
\end{Remark}

The weak saturation numbers for the following special caterpillars can also be determined.  
\begin{prop}\label{3cat}
Let $T=C_{a_1,0,a_2}$, where $a_2\geq a_1>0$ and note it has $a_1+a_2+3$ vertices. 
Then 
$ \text{w-sat}(T)=v(T)-2+{{a_1}\choose 2}=a_1+a_2+1+{{a_1}\choose 2}$.
\end{prop}
\begin{proof}
The upper bound follows from Theorem 21, item (5) in~\cite{faudree2013weak}. For the lower bound, suppose $G$ is a weakly $T$-saturated graph. 
Note that there exist vertices $v_1,v_2\in V(G)$ such that 
together they are adjacent to at least $v(T)-2$ vertices. 
Let $H=G[V(G)\setminus \{v_1, v_2\}]$. 

We \textbf{claim}
that there exists a vertex $w\in V(H)$ with $d_H(w)\geq a_1-1$. 
Assume for contradiction it is not true. 
Suppose that in the $T$-bootstrap percolation process, 
the first edge added to $H$  is $e=w_1w_2$, which creates a copy of $T$.
In that copy, at least one of these two ends of $e$, say $w_1$, 
has degree at least $a_1+1$. Initially before adding the edge, by assumption, 
$d_H(w_1)\leq a_1-2$. It follows that  $v_1, v_2, w_2$, and all vertices in the initial $N_H(w_1)$ must be  adjacent to $w_1$ in the copy. 
Nevertheless, any vertex $u$ not listed above must belong to $V(H)$,
 and therefore $d_G(u)\leq d_H(u)+2\leq a_1$, no vertex can serve as the other vertex whose degree is at least $a_1+1$ in the copy, leading to a contradiction. 

Following the claim, we obtain a vertex  $w\in V(H)$ 
with $d_H(w)\geq a_1-1$. Let $H_1=H[V(H)\setminus \{w\}]$. Using similar  arguments, we show that there exists a vertex $w'$ in $H_1$ such that $d_{H_1}(w')\geq a_1-2$. Inductively, we can choose a sequence of vertices and eventually it implies that 
$e(H)\geq (a_1-1)+(a_1-2)+\dots+1={{a_1}\choose 2}$. Thus, $e(G)\geq v(T)-2+e(H)\geq v(T)-2+{{a_1}\choose 2} $.   
\end{proof}

The following is an immediate corollary. 
\begin{coro} For every tree $T$ with diameter at most $4$, $\text{w-sat}(T)$ is  determined.
\end{coro}
\begin{proof} 
If $T$ has diameter $2$, then $T$ is the star graph $S_{K}$ on $K$ vertices for some $K$, and $\text{w-sat}(T)={K-1\choose 2}$.
If $T$ has diameter $3$, then $\text{w-sat}(T)$ is determined by Theorem~\ref{goodcat}.
 If $T$ has diameter $4$, there is a unique vertex $v_0\in V(T)$, such that for any $u\in V(T)$, the distance between $u$ and $v_0$ is no more than $2$. 
 If $v_0$ is adjacent to leaves, we can verify that $T$ meets condition (b) of Remark~\ref{extend}; thus, $\text{w-sat}(T)$ is given. If $v_0$ is not adjacent to any leaves, using the same technique as in the proof of Proposition~\ref{3cat}, we can show that $\text{w-sat}(T)=v(T)-2+{\delta_e(T) \choose 2}$.
\end{proof}

\begin{coro}\label{secondlargest}
Let $S_{K}$ denote the star graph with $K$ vertices.
\begin{enumerate}[(a)]
\item Let $T$ be a tree with $K=2\ell$ vertices, and 
$T\neq S_{K}$. Then 
\begin{equation}\label{double_star}
\text{w-sat}(T)\leq \frac{\ell^2}{2}-\frac \ell 2+1.
\end{equation}
Moreover, the equality holds if and only if $T=C_{\ell-1,  \ell-1}$ or $T=C_{\ell-1,0,\ell-2}$.
 \item Let $T$ be a tree with  $K=2\ell+1$ vertices and 
 $T \neq S_{K}$.  Then
 \begin{equation}\label{tri_star}
 \text{w-sat}(T)\leq \frac{\ell^2}{2}+\frac {\ell}{2}.
 \end{equation}
 Moreover, the equality holds if and only if $T =C_{\ell-1,0,\ell-1}$.
 \end{enumerate}
\end{coro}

\begin{proof}
 \noindent{(a)}. Let $P$ be a path contained in $T$ with maximum length.
Since $T \neq S_{K}$, the two vertices $v_1$ and $v_2$,  which are adjacent to the two ends of $P$, are distinct. 
Both vertices $v_1$ and $v_2$ 
have all but one of their neighbors are leaves. It follows that $d(v_1)+d(v_2)\leq K$. 
Assuming $d(v_1)\leq d(v_2)$, it follows that $d(v_1)\leq \frac{K}{2}=\ell$. 
$d(v_1)=\ell$ if and only if $T=C_{\ell-1,\ell-1}$. 
Moreover, $d(v_1)=\ell-1$ exactly in one of the following three cases:
$T=C_{\ell-1,0,\ell-2}, T=C_{\ell-2,\ell}$, or $T=C_{\ell-2,0,0,\ell-2}$.
Following Theorem~\ref{goodcat}, Proposition~\ref{3cat}, as well as Theorem 21, item (6) of~\cite{faudree2013weak}, we know that 
 $\text{w-sat}(C_{\ell-1, \ell-1})= \text{w-sat}(C_{\ell-1,0,\ell-2})
 =K-2+{{\ell-2}\choose 2}=\frac{\ell^2}{2}-\frac \ell2 +1$, 
  $\text{w-sat}(C_{\ell-2,\ell})=K-2+{\ell-3 \choose 2}$,  and 
  $\text{w-sat}(C_{\ell-2,0,0, \ell-2})=K-2$.
  In the remaining cases, since $\delta_e(T)<\ell-2$,
 we obtain $\text{w-sat}(T)< K-2+{ {\ell-2} \choose 2}$ by Lemma~\ref{d_endstar}.

\noindent{(b).} 
Again, let $P$ be a longest path in $T$ and let $v_1$ and $v_2$ be two vertices  adjacent to the endpoints of $P$, respectively. 
Both of them have all but one of their 
neighbors as leaves. 
Then it follows that 
$d(v_1)+d(v_2)\leq K=2\ell+1$. 
Assuming $d(v_1)\leq d(v_2)$, 
we conclude that $d(v_1)\leq \ell$, with equality if and only if  $T=C_{\ell-1, 0, \ell-1}$ or $T=C_{\ell-1,\ell}$.
By Theorem~\ref{goodcat} and Proposition~\ref{3cat}, 
$\text{w-sat}(C_{\ell-1,\ell}) =K-2+{\ell-2 \choose 2}$
and  $\text{w-sat} (C_{\ell-1,0,\ell-1}) = K-2+{{\ell-1}\choose 2} = \frac{\ell^2}{2}+\frac \ell2$.
Finally, for all remaining cases, since $\delta_e(T)\leq \ell-2$, Lemma~\ref{d_endstar} implies that
 $\text{w-sat}(T) \leq K-2+{{\ell-2}\choose 2}$. 
\end{proof}

\section{Local Structures of Good trees}
\label{sec:Local_Structures}
Recall that a tree $T$ is a {\it good tree} if $\text{w-sat}(T)=e(T)-1$. 
In this section, we will discuss the structure of a good tree. 
The following sufficient condition is a direct consequence of Lemma~\ref{d_endstar}, so we omit the proof. 

\begin{lemma}\label{P2_condition} 
Let $T$ be a tree with $\delta_e(T)=1$, that is, it contains a leaf whose neighbor has degree $2$. Then, $T$ is a good tree.
\end{lemma}

\begin{prop} Let $T$ be a tree on $K$ vertices, where $K\geq 6$. Specify six vertices, namely $v_1,\ldots,v_6$ and assume that $T$ has the following local structure:
$d(v_1)=d(v_2)=1$, $d(v_3)=d(v_4)=3$, and $N(v_3)=\{v_1,v_4,v_5\}$ 
and $N(v_4)=\{v_2,v_3,v_6\}$. Then $T$ is a good tree.
\end{prop}

\begin{proof}
Let $H$ be a copy of $T$, and we denote by $u_i$ the embedding image of $v_i$, for $i=1,\ldots,6$. Suppose $n\geq 2K-1$, and
define $G= H\setminus e + \overline{K}_{n-K}$, where $e$ is any edge of $T$. 
We add edges to $G$ in the following order:
\begin{enumerate}[(a)]
\item Add $e$. 
\item For any $u\in V(G)\setminus V(H)$, 
add $uu_3$ to $G$; $T$ can be embedded such that the edge $v_1v_3$ is sent to $uu_3$. Similarly, we can add $uu_4$ to $G$. 
\item For any $u\in V(G)\setminus V(H)$, add $uu_5$ to $G$; $T$ can be embedded such that the edges $v_3v_5$, $v_3v_4$, $v_3v_1$ are sent to $uu_5$, $uu_3$, $uu_4$, respectively. Similarly, we can add $uu_6$ to $G$. 
\item  For any two distinct vertices $u'$, $u"\in (V(G)\setminus V(H))$, 
add $u'u"$ to $G$; $T$ can be embedded such that 
the vertices  $v_1$, $v_2$, $v_3$, $v_4$, $v_5$, $v_6$,
are sent to $u_3$, $u_4$, $u'$, $u''$, $u_5$, $u_6$, respectively. 
\item Note that we obtain a complete graph of size at least $n-K+1$, which, by Lemma~\ref{min_degree_condition}, is weakly $T$-saturated, and the remaining edges can be added.
\end{enumerate}
\end{proof}

The next proposition is a generalization of Theorem 20, item (1) of~\cite{faudree2013weak}, and the proof is also similar.

\begin{prop}
Let $T$ be a tree with an induced path $v_1v_2\cdots v_{2k}$, where $k\geq2$. If $T$ satisfies $d(v_2)=d(v_3)=\dots=d(v_{2k-1})=2$, and $v_1$ and $v_{2k}$ are adjacent to leaves, then $T$ is a good tree.   
\end{prop}

\begin{proof}
Let $H$ be a copy of $T$, and denote by $u_i$ the embedding image of $v_i$ for $i=1,\ldots,2k$.  Suppose $n\geq 2v(T)$, and
define $G= H\setminus e + \overline{K}_{n-v(T)}$, where $e$ is any edge of $T$. 
We add edges to $G$ in the following order: 
\begin{enumerate}[(a)]
\item Add $e$. 
\item Add $wv_1$ for each vertex $w\in V(\overline{K}_{n-v(T)})$ to replace a pendent edge at $v_1$.
\item If $k>2$, add $wv_3$ for any vertex $w\in V(\overline{K}_{n-v(T)})$ to replace $v_2$ with $w$. Similarly, we can add $wv_{2\ell-1}$ for all $1\leq \ell\leq k-1$.
\item For any vertices $w_1$ and $w_2$ in $V(\overline{K}_{n-v(T)})$, add $w_1w_2$ to replace the path $v_{2k-3}v_{2k-2}v_{2k-1}v_{2k}$ with $v_{2k-3}w_1w_2v_{2k}$.
\item Note that we obtain a complete graph of size at least $n-v(T)$, which, by Lemma~\ref{min_degree_condition}, is weakly $T$-saturated, and the remaining edges can be added.
\end{enumerate}
\end{proof}

We now prove Theorem~\ref{goodtree}, which states the sufficient condition given in Lemma~\ref{P2_condition} becomes necessary for certain families of graphs.

\begin{theorem}[Theorem~\ref{goodtree} restated] Suppose $T$ is a tree on $K$ vertices, where $K \geq 4$, and the distances between any two leaves are even.
Then, $T$ is a good tree 
if and only if there exists a leaf  adjacent to a vertex of degree 2.
\end{theorem}

\begin{proof}
By Proposition~\ref{P2_condition}, if there exists a vertex of degree 2  adjacent to a leaf, then $\text{w-sat}(n,T_{K})=K-2$.

Now, we only need to show that if $T$ satisfies the following two properties, (P1) and (P2), 
then $T$ is not a good tree. 

\noindent{\textbf{(P1).}} The distances between any two leaves are all even.
 
\noindent{\textbf{(P2).}} There is no leaf adjacent to a vertex of degree 2. 

For any tree satisfying (P1), we properly color $V(T)$ red and blue, assuming that the leaves are always red. Let $R(T)$ be the set of red vertices in $V(T)$, and $B(T)$ be the set of blue vertices. First, we have the following claim. 

\noindent{\textbf{Claim 1.}} If $T$ satisfies (P1), then $|R(T)|>|B(T)|$.
\begin{proof}[Proof of Claim 1]
Take a non-leaf vertex as an ancestor, namely $v_0$. Let $$L_k=\{v|d(v,v_0)=k\},\  k=1,2,\dots.$$

If the distances between $v_0$ and the leaves are odd, note that the vertices in $L_{2t}$ are non-leaves, and each of them has at least one child in $L_{2t+1}$; hence $|L_{2t}|\leq |L_{2t+1}|$. Also, note that $|L_0|<|L_1|$ since $v_0$ is non-leaf. Summing these inequalities yields Claim 1.

If the distances between $v_0$ and the leaves are even, similarly we have $|L_{2t-1}|\leq |L_{2t}|$ and $1\leq|L_0|$. Summing these inequalities yields Claim 1.
\end{proof}

Furthermore, if $T$ satisfies (P1) and (P2), we have the following claim. Note that (P2) can also be restated as follows: (P2') if a non-leaf vertex $v$ has at most one non-leaf neighbor, then $d(v)>2$. 

\noindent{\textbf{Claim 2.}} For any tree $T$ that satisfies both (P1) and (P2') and is not $S_3$ or $S_4$, we have $R(T)\geq B(T)+3$.
\begin{proof}[Proof of Claim 2]
If $T=S_{K}$ with $K\geq 5$, then $R(T)=B(T)+K-2\geq B(T)+3$. 
If $T$ is not a star, then there are at least two non-leaves $v_1,v_2\in B(T)$ satisfying property (P2'). That is,  
each of $v_1$ and $v_2$ has exactly one non-leaf neighbor. 
By (P2'), we can delete a leaf adjacent to each of $v_1$ and $v_2$ to obtain a new tree $T'$, which still satisfies (P1). Hence,  $R(T')\geq B(T')+1$, and $R(T)=R(T')+2\geq B(T')+3=B(T)+3$. 
In both cases, we conclude that $R(T)\geq B(T)+3$.
\end{proof}
Fix a large integer $n$ and let $H$ be an isomorphic copy of $T$. 
Define $G = H + \overline{K}_{n-K}$.
We regard $G$ as a bipartite graph with a vertex bipartition 
$V(G) = V_1 \cup V_2$, 
where
$V_1$ represents the isomorphic image of the blue vertices of $T$, and 
$V_2$ consists of all the remaining vertices.
By definition, $V_1$ and $V_2$ are independent sets initially.

Assume, to the contrary, that $T$ is a good tree. Then we can continue adding new edges to $G$, creating new copies of $T$ at each step, and eventually arrive at $K_n$. 
Let $e=uv$ be the first edge added inside one of the parts $V_1$ or $V_2$, which creates a new copy $T'$ of $T$.  
Thus, $T'\setminus e$ is a union of two trees $T_1$ and $T_2$. 
We can define a vertex bipartition 
$X \cup Y$ for the graph $T_1+T_2$ 
such that 
$u,v \in X$, where $X$ is a subset of $V_1$ or $V_2$, and $Y$ is a subset of the other. 

\noindent{\textbf{Claim 3.}}
\begin{equation}\label{B'R'B}
\min\{ |X|, |Y|\}  > |B(T)|=|V_1|.
\end{equation} 
\begin{proof}[Proof of Claim 3.]

First, by Claim 2, $|R(T)|\geq |B(T)|+2$. 

If $e$ is a pendant edge of $T'$, then clearly $|X|=|B(T)|+1$, $|Y|=|R(T)|-1$, and the conclusion 
follows.
So assume $e$ is not a pendant edge. Without loss of generality, assume that in $T'$, $u\in V(T_1)$ is at even distances to all leaves
and  $v\in V(T_2)$ is at odd distances to all the leaves. Hence $u$ is red and $v$ is blue.
We obtain a bipartition $V(T_1)=B_1 \cup R_1$, such that $u \in R_1$, then $R_1\subset X$, $B_1\subset Y$, all the vertices in $R_1$ are colored red, and those in $B_1$ are colored blue.
Note that $T_1$ contains at least $3$ vertices 
and satisfies property (P1). 
It follows that $|R_1| \geq |B_1|+1$ by Claim 1. 
Similarly, we obtain the vertex bipartition $V(T_2) = B_2\cup R_2$, 
such that $v\in R_2$, and we have $R_2\subset X$, $B_2\subset Y$, vertices in $R_2$ and $B_2$ are colored blue and red respectively.
If $v$ is a leaf in $T_2$, then 
$T_2$ cannot be a single edge, 
because otherwise, $T'$ does not satisfy (P2).
By adding two more leaves to $T_2$, we obtain 
$T_2'=T_2\cup \{vu_1,vu_2\}$, which is not a copy of 
$S_4$ and 
satisfies both properties (P1) 
and (P2). 
The pair $B_2\cup \{u_1,u_2\}$ and $R_2$ 
can be regarded as a bipartition for the tree $T_2'$. Note that $B_2\cup \{u_1,u_2\}$ is a set including all the leaves of $T_2'$. By Claim 2, it follows that 
$|B_2|+2\geq |R_2|+3$, which means that $|B_2| \geq |R_2|+1$. 
If $v$ is not a leaf in $T_2$, then $T_2$ satisfies (P1) and we conclude again 
$|B_2|\geq |R_2|+1$.

By the definitions of $X$ and $Y$, we have that $Y=B_1 \cup B_2$ and $X=R_1 \cup R_2$
and that $B = B_1 \cup R_2$ and $R = R_1 \cup B_2$.
Moreover, $|Y|=|B_1|+|B_2|\geq |B_1|+|R_2|+1=|B(T)|+1$ and $|X|=|R_1|+|R_2|\geq |B_1|+|R_2|+1=|B(T)|+1$, which implies Claim 3. 
\end{proof}

Now note that Claim 3 contradicts with the fact that one of $X$ and $Y$ is a subset of $V_1$, and we conclude that $T$ is not a good tree.
\end{proof}

We now construct a counterexample of Theorem 8 in~\cite{faudree2013weak}. Consider a tree $T$ on vertex set $\{v_i:1\leq i\leq 7\}$ with edge set $\{v_1v_2, v_1v_3,v_2v_4,v_2v_5,v_3v_6,v_3v_7\}$, and note that all the leaves of $T$ are of distance $2$ from $v_1$, and no $2$-degree vertex in $T$ is adjacent to a leaf. By the above result, $T$ is not a good tree. However, $T$ has a good endtree, namely $T[\{v_1,v_2,v_3,v_4,v_5\}]$, contradicting with Theorem 8 in~\cite{faudree2013weak}.

One may note that, each of the local structures given above contains a non-pendent vertex with small degree. However, we will show that this is not necessary for a good tree.

\begin{prop}
For any positive integer $N$, there exists a good tree $T$ whose second smallest degree is greater than $N$.
\end{prop}

\begin{proof}
We can construct $T$ in this way: Let $v_0$ be a root and let $L_k=\{v|d(v, v_0)=k\}$, $k=1,\dots,5$, such that $v_0$ has $N+1$ children and every vertex in $L_k$, $k=1, \dots, 4$, has $N$ children. Then the vertices of $T$ are of degree $1$ or $N+1$. Let $v_1$ be a vertex in $L_1$, and $v_2$ be a child of $v_1$ in $L_2$. We add a pendent edge to every non-pendent vertex in $T$ except $v_1$ and $v_2$. We still denote the resulting tree by $T$. Directly apply Proposition~\ref{endd-mind} to $v_1$ and $v_2$ taking the roles of $u$ and $w$ respectively, note that, by our construction, any neighbor of $v_1$ or $v_2$ is either adjacent to a leaf or $v_1$ or $v_2$ itself, we conclude that $|S|=0$ and $\text{w-sat}(T)=e(T)-1$.
\end{proof}

\bibliographystyle{plain}
\addcontentsline{toc}{chapter}{Bibliography}
\bibliography{tree3}

\end{document}